\documentclass[12pt]{amsart}
\usepackage{color}
\usepackage{amsmath,amsfonts,amsthm,amssymb} 
\usepackage{graphicx}
\usepackage{comment}
\usepackage{hyperref}
\usepackage{cite}
\newtheorem{theorem}{Theorem}
\newtheorem{lemma}[theorem]{Lemma}
\newtheorem{corollary}[theorem]{Corollary}
\newtheorem{proposition}[theorem]{Proposition}
\theoremstyle{definition}
\newtheorem{remark}[theorem]{Remark}

\usepackage[us,12hr]{datetime}
\numberwithin{equation}{section} \numberwithin{theorem}{section}
\newcounter{stepctr}
{\end{list}}

\def\XXint#1#2#3{{\setbox0=\hbox{$#1{#2#3}{\int}$}
     \vcenter{\hbox{$#2#3$}}\kern-.5\wd0}}

\newcommand{\mc}[1]{\mathcal{#1}}
\newcommand{\mbb}[1]{\mathbb{#1}}

\DeclareMathOperator{\tr}{tr}

\newcommand{\M}{\mathcal{M}}

\newcommand{\R}{\mbb{R}}

\newcommand{\pd}{\partial}
\newcommand{\cd}{\nabla}

\date{\today, \currenttime}

\protected\def\vts{%
  \ifmmode
    \mskip0.5\thinmuskip
  \else
    \ifhmode
      \kern0.08334em
    \fi
  \fi
}

\newcommand{\lb}{\left(}
\newcommand{\rb}{\right)}
\newcommand{\lsb}{\left[}
\newcommand{\rsb}{\right]}

\def\labelitemi{--}
\def\ba #1\ea {\begin{align} #1\end{align}}
\def\bann #1\eann {\begin{align*} #1\end{align*}}
\def\ben #1\een {\begin{enumerate} #1\end{enumerate}}
\def\bi #1\ei {\begin{itemize}\renewcommand\labelitemi{--} #1\end{itemize}}

\newcommand{\inner}[2]{\left\langle#1,#2\right\rangle} 

\address{Department of Mathematics, University of Tennessee Knoxville, Knoxville TN, USA, 37996-1320}
\address{School of Mathematical and Physical Sciences, The University of Newcastle, Newcastle, NSW, Australia, 2308}
\email{mlangford@utk.edu}
\email{mathew.langford@newcastle.edu.au}

\address{School of Mathematical Sciences, Queen Mary University of London, Mile End Road, London, UK, E1 4NS}
\email{h.nguyen@qmul.ac.uk}

\begin{document}
\title[Sharp pinching estimates for MCF in the sphere]{Sharp pinching estimates for mean curvature flow in the sphere}
\author{Mat Langford}
\author{Huy The Nguyen}
\address{}
\email{}
\begin{abstract} 
We prove a suite of asymptotically sharp quadratic curvature pinching estimates for mean curvature flow in the sphere which generalize Simons' rigidity theorem for minimal hypersurfaces. We then obtain derivative estimates for the second fundamental form which we utilize, via a compactness argument, to obtain a convexity estimate. Together, the convexity and cylindrical estimates yield a partial classification of singularity models. We also obtain new rigidity results for ancient solutions.
\end{abstract}
\maketitle

\tableofcontents

\section{Introduction}

Given $n\ge 2$ and $m\le\lceil\frac{n}{2}\rceil$, we study the evolution by mean curvature of hypersurfaces of $\mathbb S_K^{n+1}$, the round sphere in $\R^{n+2}$ of sectional curvature $K$, satisfying the quadratic pinching condition
\begin{equation}\label{eq:strict quadratic pinching}
\begin{cases}
\vert A\vert ^2<\frac{3}{4}H^2+\frac{4}{3}K& \text{if}\; n=2\; \text{and}\; m=1\,,\\
\vert A\vert ^2<\frac{3}{5}H^2+\frac{8}{3}K& \text{if}\; n=3\; \text{and}\; m=2\,,\\
\vert A\vert ^2<\frac{1}{n-m}H^2+2mK& \text{if}\; n\ge3\;\text{and}\; m\le \lfloor \frac{n}{2}\rfloor\,,\\
\vert A\vert ^2<\frac{2}{n}H^2+nK& \text{if}\; n\ge 4\; \text{and}\; m=\lceil \frac{n}{2}\rceil\,,
\end{cases}
\end{equation}
where $A$ is the second fundamental form and $H$, its trace, is the mean curvature. The first case and the third case with $m=1$ were treated in \cite{Hu87} (cf. \cite{An02}), while the second case and the third case with $m=2$ were treated in \cite{LangfordNguyen}. In the former, it was shown that the flow preserves the condition and drives solutions either to a ``small $\mathbb S^n$'' in finite time or to a ``large $\mathbb S^n$'' in infinite time. In the latter, it was shown that the flow preserves the condition and decomposes the solution, via surgery on ``necks'' (following Huisken and Sinestrari \cite{HuSi09}), into a finite number of components, each of which is either a small $\mathbb S^n$, a large $\mathbb S^n$, or $\mathbb S^1$ times a small $\mathbb S^{n-1}$. Here, we extend some of the methods of \cite{Hu87,LangfordNguyen} to treat the entire class of conditions \eqref{eq:strict quadratic pinching}. Namely, we show that \eqref{eq:strict quadratic pinching} is preserved and improves, becoming sharp when either the curvature or the time of existence becomes large. These estimates are analogues of \cite[Theorem 5.2]{HuSi15} and \cite[Corollary 1.2]{L17} (cf. \cite[Theorem 1.1]{LaLy}).

\subsection*{Acknowledgements}  
M.~Langford was supported by an Australian Research Council DECRA fellowship. H.~T.~Nguyen was supported by the EPSRC grant EP/S012907/1.

\section{Preserved curvature conditions}

\subsection{Quadratic curvature condition}
If the strict quadratic curvature inequality \eqref{eq:strict quadratic pinching} holds on a hypersurface of $\mathbb{S}_K^{n+1}$, then we can find some $\alpha>0$ such that
\begin{align}\label{eq:uniform quadratic pinching}
|A|^{2}\leq\frac{1}{n-m+\alpha}H^{2}+2(m-\alpha)K\,.  
\end{align}
Without loss of generality, $\alpha\in(0,1)$. When $m\le \min\left\{\tfrac{n}{2},\tfrac{2(n-1)}{3}\right\}$, this inequality is preserved under mean curvature flow for any $\alpha\in(0,1)$. When $m\ge \min\left\{\tfrac{n}{2},\tfrac{2(n-1)}{3}\right\}$, the inequality is preserved if $\alpha>m-\min\{\frac{n}{2},\frac{2(n-1)}{3}\}$.

\begin{theorem}[Cf. {\cite[1.4 Lemma]{Hu86}}]
Let $X:\mc M^n\times[0,T)\to \mbb S_K^{n+1}$ be a solution to mean curvature flow such that \eqref{eq:uniform quadratic pinching} holds on $\mc M^n\times\{0\}$ for some $\alpha\in(0,1)$. If $m<\alpha+\min\{\frac{n}{2},\frac{2(n-1)}{3}\}$, then \eqref{eq:uniform quadratic pinching} holds on $\mc M^n\times\{t\}$ for all $t\in[0,T)$.
\end{theorem}
\begin{proof}
Suppose that \eqref{eq:uniform quadratic pinching} holds on the initial hypersurface for some $m\in\{1,\dots,n\}$ and some $\alpha\in(0,1)$. If $m<\alpha+\min\{\frac{n}{2},\frac{2(n-1)}{3}\}$, then
\[
\frac{1}{n-m+\alpha}<\frac{3}{n+2}\;\;\text{and}\;\; \frac{n}{2n-2(m-\alpha)}<1\,.
\]
Set
\[
a_{m}:=\frac{1}{n-m+\alpha}\;\; \text{and}\;\; b_{m}:=2(m-\alpha)\,.
\]
Using the evolution equations for $H^2$ and $\vert A\vert^2$ (see, for example, \cite[Equations (2.11) and (2.12)]{LangfordNguyen}), we obtain
\begin{align*}
(\partial_{t}-\Delta)\!\left(|A|^{2}-a_{m} H^{2}\right)={}&- 2 \left(|\nabla A|^{2} 
- a_{m}|\nabla H|^{2}\right) + 2 b_{m} K ( |A|^{2} + nK) \\
&+ 2 ( |A|^{2} - a_{m} H^{2} - b_{m} K ) ( | A|^{2} + n K )\\
{}& - 4 nK \left(|A|^{2} - \tfrac{1}{n}H^{2}\right).
\end{align*}
Since $\frac{2}{2n-b_{m}}=a_m$ and $\frac{n}{2n-b_{m}}\leq 1$, we can estimate
\begin{align*}
2b_{m} K (|A|^{2} + nK)-4 nK&\left(|A|^{2}- \tfrac{1}{n} H^{2} \right)\\
={}& 2 K\left((b_{m}- 2n)|A|^{2} + 2 H^{2} + b_{m} n K\right) \\
={}&-2K( 2n-b_{m})\left(|A|^{2} - \tfrac{ 2}{2n - b_{m}} H^{2} - \tfrac{ nb_{m} }{ 2n - b_{m}} K\right)\\
\leq{}&-2K(2n-b_{m})\left(|A|^{2}-a_mH^{2}-b_{m}K\right)\,.
\end{align*}
Estimating $a_m \leq \frac{3}{n+2}$ and applying the Kato-type inequality $\vert \cd A\vert^2\ge \frac{3}{n+2}\vert \cd H\vert^2$, we arrive at
\begin{align*}
(\partial_{t}-\Delta)\big(|A|^{2}-a_{m} H^{2}-{}&b_mK\big)\\
\leq{}& 2\left(|A|^{2}+(b_m-2n)K\right)\left(|A|^{2}-a_{m} H^{2}-b_{m}K\right).
\end{align*}
The claim now follows from the maximum principle.
\end{proof}

Note that
\[
\min\left\{\tfrac{n}{2},\tfrac{2(n-1)}{3}\right\}=\begin{cases} \frac{2(n-1)}{3}=\frac{2}{3}, \frac{4}{3} & \text{when}\;\; n=2,3\\
\frac{n}{2} & \text{when}\;\; n=4,\dots.
\end{cases}
\]
The non-vacuous cases are therefore:
\begin{itemize}
\item $n=2$: $m=1$ and $\alpha>\frac{1}{3}$.
\item $n=3$: $m=1$ and $\alpha>0$, or $m=2$ and $\alpha>\frac{2}{3}$.
\item $n=2k\ge 4$: $m=1,\dots,\frac{n}{2}$ and $\alpha>0$.
\item $n=2k+1\ge 5$: $m=1,\dots,k$ and $\alpha>0$, or $m=k+1$ and $\alpha>\frac{1}{2}$.
\end{itemize}


\subsection{Rigidity}

The quadratic curvature condition \eqref{eq:strict quadratic pinching} is optimal for $m$-cylindrical estimates when $n\ge 4$ and $m\le \frac{n}{2}$. Indeed, consider the hypersurfaces $\mc M^{k, n-m}(r,s) = \mathbb S^{m}(r) \times \mathbb S^{n-m}(s)$ with $r^{2} + s^{2} = 1$, where $\mathbb S^m(r)$ is the $m$ dimensional sphere of radius $r$. The second fundamental forms have eigenvalues $\lambda$ with multiplicity $m$ and $\mu$ with multiplicity $n-m$ and $ \lambda \mu = -1$. Observe that
\begin{align*}
 |A|^{2} & = \frac{ms^{4} + (n-m) r^{4} } { r^{2}s^{2}},
\end{align*}
and
\begin{align*}
 H
 & = \frac{(n-m) r^{2} - ms ^{2} } { rs},
\end{align*}
so that
\begin{align*}
 H^{2} & = \frac{(n-m)^2r^{4} + m^2s^{4} - 2m(n-m) r^{2} s^{2} } { r^{2} s^{2} },
\end{align*}
which then yields
\begin{align*}
|A|^2-\frac{1}{n-m}H^2-2m= \frac{m(n-2m)}{(n-m)} \frac{s^2}{r^2}. 
\end{align*}
Thus, for any $\varepsilon>0$, there is a submanifold of the topological form $ \mathbb{S}^m \times \mathbb{S}^{n-m}$ satisfying 
\begin{align*}
|A| ^2-\frac{1}{n-m}H^2-2m\le \varepsilon\,.
\end{align*} 


\subsection{A class of hypersurfaces}

Given $n\ge 2$, $m\le n$, $K>0$, $\alpha\in(0,1)$, $V<\infty$ and $\Theta<\infty$ such that $m-\alpha<\min\{\frac{n}{2},\frac{2(n-1)}{3}\}$, we shall work with the class $\mathcal{C}^n_m(K,\alpha,V,\Theta)$ of hypersurfaces $X:\M\to\mathbb{S}^{n+1}_K$ satisfying
\begin{enumerate}
\item $\displaystyle \max_{\mathcal{M}^n\times\{0\}}\big(|A|^2-\tfrac{1}{n-m+\alpha}H^2\big)\le 2(m-\alpha)K^2$,
\item $\displaystyle \mu_0(\mathcal{M}^n)\le VK^{-\frac{n}{2}}$, and
\item $\displaystyle \max_{\mathcal{M}^n\times\{0\}}H^2\leq \Theta K$,
\end{enumerate}
where $\mu_t$ is the measure induced by $X(\cdot,t)$. Every properly immersed hypersurface of $\mathbb{S}^{n+1}_K$ which satisfies the strict quadratic pinching condition \eqref{eq:strict quadratic pinching} lies in the class $\mathcal{C}^n_m(m,\alpha,V,\Theta)$ for some choice of parameters $\alpha$, $V$, and $\Theta$.

\section{Cylindrical estimates}\label{ssec:cylindrical estimate}
 
\begin{theorem}[Cylindrical estimates (Cf. \cite{Hu84,Hu87,HuSi09,LangfordNguyen,Ng15})]\label{thm:cylindrical estimate}
Let $X:\M^n\times[0,T)\to\mathbb{S}_K^{n+1}$ be a solution to mean curvature flow with initial condition in the class $\mathcal{C}_m^n(K,\alpha,V,\Theta)$. There exist $\delta=\delta(n,m,\alpha)>0$, $\eta_0=\eta_0(n,m,\alpha)>0$ and, for every $\eta\in(0,\eta_0)$, $C_\eta=C_\eta(n,\alpha,V,\Theta,\eta)<\infty$ such that
\begin{align}\label{eq:cylindrical estimate}
|A|^2-\frac{1}{n-m+1}H^2 \leq \eta H^2 + C_\eta K\vts\mathrm{e}^{-2\delta Kt} \quad\text{in}\quad \mc M^n\times[0,T)\,.
\end{align}
\end{theorem}

Before proving Theorem \ref{thm:cylindrical estimate}, let us mention some immediate implications. First observe that, when $n\ge 4$ and $m=\lceil \frac{n}{2}\rceil$, we recover Simons' theorem \cite{Si68}: any \emph{minimal} hypersurface of $\mathbb{S}_K^{n+1}$ satisfying $\vert A\vert^2<nK$ is totally geodesic. 

Next observe that any \emph{ancient} solution $X:\M\times(-\infty,0)\to \mathbb{S}^{n+1}_K$ satisfying the uniform quadratic pinching condition \eqref{eq:uniform quadratic pinching} (for suitable $m$ and $\alpha$) and uniform area and curvature bounds as $t\to-\infty$ will satisfy
\[
\vert A\vert^2-\frac{1}{n-m+1}H^2\le 0<2(m-1)
\]
unless $m=1$. In particular, it satisfies the uniform quadratic pinching condition \eqref{eq:uniform quadratic pinching} with $m\mapsto m-1$. Iterating the argument, we find that $X:\M\times(-\infty,0)\to \mathbb{S}^{n+1}_K$ is umbilic, and hence, assuming $\M$ is connected, either a stationary hyperequator or a shrinking hyperparallel (cf.  \cite[Theorem 6.1]{HuSi15}). In fact, the uniform area and curvature bounds are superfluous since we may instead apply Corollary \ref{cor:maximum principle} below. We thus obtain the following parabolic analogue of Simons' theorem, which generalizes \cite[Theorem 6.1 (2)]{HuSi15}.

\begin{theorem}\label{thm:rigidity ancient}
Let $X:\M\times(-\infty,\omega)\to \mathbb S_K^{n+1}$ be an ancient solution to mean curvature flow. If
\begin{itemize}
\item $n=2$ and $\displaystyle \limsup_{t\to-\infty}\max_{\M\times\{t\}}(\vert A\vert^2-\tfrac{3}{4}H^2-\tfrac{4}{3}K)<0$, or
\item $n=3$ and $\displaystyle \limsup_{t\to-\infty}\max_{\M\times\{t\}}(\vert A\vert^2-\tfrac{3}{5}H^2-\tfrac{8}{3}K)<0$, or
\item $n\ge 4$ and $\displaystyle \limsup_{t\to-\infty}\max_{\M\times\{t\}}(\vert A\vert^2-\tfrac{2}{n}H^2-nK)<0$,
\end{itemize}
then $X:\M\times(-\infty,\omega)\to \mathbb S_K^{n+1}$ is either a stationary hyperequator or a shrinking hyperparallel.
\end{theorem}

Finally, we note that, by a similar argument, any immortal solution $X:\M\times[0,\infty)\to \mathbb S_K^{n+1}$ to mean curvature flow with initial datum satisfying the uniform quadratic pinching condition converges (assuming $\M$ is connected) to a stationary hyperequator. In particular, $\M$ is diffeomorphic to $\mathbb S^n$. 


Returning to the proof of Theorem \eqref{thm:cylindrical estimate}, set
\[
a:=\frac{1}{n-m+\alpha}-\frac{1}{n-m+1}+\eta_0-\eta\;\;\text{and}\;\; b:=2(m-\alpha)\,,
\]
where $\eta_0=\eta_0(n,m,\alpha)$ is the largest number satisfying
\[
n\eta_0\le 
\begin{cases}
\frac{n}{2}+\alpha-m\\
1-\frac{n+2}{3}\left(\frac{1}{n-m+\alpha}+\eta_0\right).
\end{cases}
\]
Note that, by hypothesis, $\eta_0>0$.

We will prove the estimate \eqref{eq:cylindrical estimate}, with $\delta=n\eta_0$, by obtaining a bound for the function
\begin{align*}
f_{\sigma, \eta}:= \left[|A|^2-\left(\frac{1}{n-m+1}+\eta\right)H^2\right]W^{\sigma-1}
\end{align*}
for some $\sigma\in(0,1)$ and any $\eta\in \left(0,\eta_0\right)$, where $W$ is defined by
\[
W:=aH^2+bK\,.
\]
Observe that $W>0$ and $f_{\sigma,\eta}\leq W^\sigma$. The choice of the constants $a$ and $b$ is determined by the need to obtain the good gradient and reaction terms in the following lemma. 

\begin{lemma}\label{lem:evol_fsigma}
There exists $\delta=\delta(n,\alpha)>0$ such that
\begin{align}\label{eq:ineqn_evol}
(\partial_t-\Delta)f_{\sigma,\eta}\leq{}&2\sigma(|A|^2+nK)f_{\sigma,\eta}-4\delta Kf_{\sigma,\eta}\nonumber\\
{}&-2\delta f_{\sigma,\eta}\frac{|\nabla A|^2}{W}+2(1-\sigma)\left\langle \nabla f_{\sigma,\eta}, \frac{\nabla W}{W}\right\rangle
\end{align}
wherever $f_{\sigma,\eta}>0$.
\end{lemma}
\begin{proof}
Set $f_\eta:=|A|^2-(\tfrac{1}{n-m+1}+\eta)H^2$. Basic manipulations (independent of the precise form of $f_\eta$ and $W$) yield
\begin{align*}
(\partial_t-\Delta)f_{\sigma,\eta}={}&W^{\sigma-1}(\partial_t-\Delta)f_\eta-(1-\sigma)f_{\sigma,\eta}W^{-1}(\partial_t-\Delta)W\\
{}&+2(1-\sigma)\left\langle\nabla f_{\sigma,\eta},\frac{\nabla W}{W}\right\rangle-\sigma(1-\sigma)f_{\sigma,\eta}\frac{|\nabla W|^2}{W^2}\,.
\end{align*}
The final term will be discarded.

Applying the evolution equations for $H^2$ and $\vert A\vert^2$ \cite[Equations (2.11) and (2.12)]{LangfordNguyen}, we compute
\begin{align*}
(\partial_t-\Delta)W=2\left(|A|^2+nK\right)(W-bK)-2a|\nabla H|^2
\end{align*}
and
\begin{align}\label{eq:evolve_f_eta}
(\partial_t-\Delta)f_\eta={}&2\left(|A|^2+nK\right)f_\eta-4nK\left(|A|^2-\tfrac{1}{n}H^2\right)\nonumber\\
{}&-2\left(|\nabla A|^2-\left(\tfrac{1}{n-m+1}+\eta\right)|\nabla H|^2\right)\,.
\end{align}

Combining these preceding three identities and estimating $f_{\sigma,\eta}\leq W^{\sigma}$, $|\nabla H|^2\leq \frac{n+2}{3}|\nabla A|^2$ and $\sigma<1$ yields
\begin{align*}
(\partial_t-\Delta)f_{\sigma,\eta}\leq{}&2\sigma(|A|^2+nK)f_{\sigma,\eta}+2(1-\sigma)\left\langle\nabla f_{\sigma,\eta},\frac{\nabla W}{W}\right\rangle\\
{}&+2W^{\sigma-1}\left(bK(|A|^2+nK)\frac{f_\eta}{W}-2nK(|A|^2-\tfrac{1}{n}H^2)\right)\\
{}&-2f_{\sigma,\eta}\left(1-\frac{n+2}{3}\left[\frac{1}{n-m+\alpha}+\eta_0\right]\right)\frac{|\nabla A|^2}{W}\,.
\end{align*}
By the definition of $\eta_0$,
\[
2f_{\sigma,\eta}\left(1-\frac{n+2}{3}\left[\frac{1}{n-m+\alpha}+\eta_0\right]\right)\frac{|\nabla A|^2}{W}\ge 2n\eta_0f_{\sigma,\eta}\frac{|\nabla A|^2}{W}\,.
\]
%
So consider the term
\begin{align*}
Z:={}&(m-\alpha)\left(|A|^2+nK\right)\frac{f_\eta}{W}+H^2-n|A|^2\\
\leq{}&(m-\alpha)\left(\frac{1}{n-m+\alpha}H^2+(b+n)K\right)\frac{f_\eta}{W}+H^2-n|A|^2\,.
\end{align*}
Noting that
\[
\frac{m-\alpha}{n-m+\alpha}=na+\frac{n}{n-m+1}+n(\eta-\eta_0)-1\;\;\text{and}\;\;m-\alpha=\frac{b}{2},
\]
and estimating $f_\eta\le W$, we find
\begin{align*}
Z\leq{}&nf_\eta+H^2-n|A|^2\\
{}&+\left(\left[\frac{n}{n-m+1}+n(\eta-\eta_0)-1\right]H^2+\frac{b}{2}(b+n)K-nbK\right)\frac{f_\eta}{W}\\
={}&-\left(\frac{n}{n-m+1}+n\eta-1\right)H^2\\
{}&+\left(\left[\frac{n}{n-m+1}+n(\eta-\eta_0)-1\right]H^2+\left[m-\alpha-\frac{n}{2}\right]bK\right)\frac{f_\eta}{W}\\
\le{}&-\left(n\eta_0H^2+\left[\frac{n}{2}+\alpha-m\right]bK\right)\frac{f_\eta}{W}\\
\le{}&-n\eta_0f_\eta\,.
\end{align*}
The claim follows.
\end{proof}

Setting $\sigma=\eta=0$, the maximum principle yields the following.
\begin{corollary}\label{cor:maximum principle}
Every solution $X:\M\times[0,T)\to\mathbb S_K^{n+1}$ to mean curvature flow with initial datum satisfying the uniform quadratic pinching condition \eqref{eq:uniform quadratic pinching} with $\alpha>m-\min\{\frac{n}{2},\frac{2(n-1)}{3}\}$ satisfies
\[
\frac{\vert A\vert^2-\frac{1}{n-m+1}H^2}{aH^2+bK}\le \mathrm{e}^{-4\delta Kt}\,,
\]
where $\delta=\delta(n,m,\alpha)$.
\end{corollary}

We wish to bound $\mathrm{e}^{2\delta Kt}f_{\sigma,\eta}$ from above. It will suffice to consider points where $f_{\sigma, \eta}>0$. To that end, consider the function
\[
f_+
:=\max\{\mathrm{e}^{2\delta Kt}f_{\sigma,\eta},0 \}\,.
\]

We first derive an $L^p$-estimate for $f_+$ using Lemma \ref{lem:evol_fsigma} and the following proposition.

\begin{proposition}\label{prop:Poincare}
Given $n\geq 3$, $\alpha\in(0,1)$ and $\eta\in(0,\frac{1}{n-m+\alpha}-\frac{1}{n-m+1})$ there exists $\gamma=\gamma(n,\alpha,\eta)>0$ with the following property: Let $X:\mc M^n\to \mbb S^{n+1}_K$ be a smoothly immersed hypersurface and $u\in W^{2,2}(\mc M)$ a function satisfying $\operatorname{spt}u\subset U_{\alpha,\eta}$, where, introducing the functions 
\[
f_{m,\eta}:=|A|^2-\frac{1}{n-m}H^2-\eta H^2
\]
and
\[
g_{m,\alpha}:=|A|^2-\frac{1}{n-m+\alpha}H^2-2(m-\alpha)K\,,
\]
the set $U_{\alpha,\eta}\subset \mathcal{M}$ is defined by
\[
U_{\alpha,\eta}:=\{x\in \mathcal{M}:f_{m-1,\eta}(x)\geq 0\geq g_{m,\alpha}(x)\}\,.
\]
For any $r\geq 1$,
\[
\gamma\int u^2W\,d\mu\leq \int u^2\left(r^{-1}\frac{|\nabla u|^2}{u^2}+r\frac{|\nabla A|^2}{W}+K\right)d\mu\,.
\]
\end{proposition}
\begin{proof}
We proceed as in \cite[Proposition 2.2]{LangfordNguyen}. By a straightforward scaling argument, it suffices to prove the claim when $K=1$. Recall Simons' identity
\[
\nabla_{(i}\nabla_{j)}A_{kl}-\nabla_{(k}\nabla_{l)}A_{ij}=\mathrm{C}_{ijkl}\,,
\]
where the brackets denote symmetrization and
\[
\mathrm{C}:=A\otimes A^2-A^2\otimes A+(g\otimes A-A\otimes g)\,.
\]
We claim that
\begin{equation}\label{eq:Poincare1}
\gamma(n,\alpha,\eta) W^3\leq |\mathrm{C}|^2+1\quad\text{in}\quad U_{\alpha,\eta}
\end{equation}
on any immersed hypersurface $X:\mathcal{M}^n\to \mathbb{S}^n$. Indeed, if this is not the case then there is a sequence of vectors $\vec\lambda^k\in \R^n$, $k\in\mathbb{N}$, satisfying
\[
f_{m-1,\eta}(\vec\lambda^k):=|\vec\lambda^k|^2-\frac{1}{n-m+1}\tr(\vec\lambda^k)^2-\eta \tr(\vec\lambda^k)^2\geq 0
\]
and
\[
g_{m,k}(\vec\lambda^k):=|\vec\lambda^k|^2-\frac{1}{n-m+\alpha}\tr(\vec\lambda^k)^2-2(m-\alpha)\leq 0\,,
\]
where $\tr(\vec\lambda):=\sum_{i=1}^n\lambda_i$, but
\[
\frac{|\mathrm{C}(\vec\lambda^k)|^2+1}{W^3(\vec\lambda^k)}\to 0
\]
as $k\to\infty$, where
\[
|\mathrm{C}(\vec\lambda)|^2:=\sum_{i,j=1}^n(\lambda_j-\lambda_i)^2(\lambda_i\lambda_j+1)^2
\]
and
\[
W(\vec\lambda):=\left(\frac{1}{n-m+\alpha}-\frac{1}{n-m+1}+\eta_0-\eta\right)\tr(\vec\lambda)^2+2(2-\alpha)\,.
\]
Set $r^2_{k}:=W(\vec\lambda^k)^{-1}\to 0$ and $\hat\lambda^k:=r_k\vec\lambda^k$. Observe that
\[
\vert \hat\lambda^k\vert\leq c(n,\alpha)
\]
and hence, up to a subsequence, $\hat\lambda^k\to \hat\lambda\in \R^n$. Computing
\[
\left(|\hat \lambda^k|^2-\frac{1}{n-m+1}\tr(\hat \lambda^k)^2\right)-\eta \tr(\hat\lambda^k)^2=r^2_k f_{m-1,\eta}(\vec\lambda^k)\geq0
\]
and
\[
\left(|\hat \lambda^k|^2-\frac{1}{n-m+\alpha}\tr(\hat \lambda^k)^2\right)-2(m-\alpha)r_k^2=r^2_k g_{m,\alpha}(\vec\lambda^k)\leq 0\,,
\]
we find
\begin{equation}\label{eq:Poincarecontra2}
\left(|\hat \lambda|^2-\frac{1}{n-m+1}\tr(\hat \lambda)^2\right)\geq \eta \tr(\hat \lambda)^2
\end{equation}
and
\begin{equation}\label{eq:Poincarecontra3}
\left(|\hat \lambda|^2-\frac{1}{n-m+\alpha}\tr(\hat \lambda)^2\right)\leq 0\,.
\end{equation}
On the other hand,
\[
\sum_{i,j=1}^n\left(\hat\lambda^k_i\hat\lambda^k_j(\hat\lambda^k_j-\hat\lambda^k_i)\right)^2 +2r_k^2\hat\lambda^k_i\hat\lambda^k_j(\hat\lambda^k_j-\hat\lambda^k_i)^2 +r_k^4(\hat\lambda^k_j-\hat\lambda^k_i)^2 = r_k^6|\mathrm{C}(\vec\lambda^k)|^2
\]
so that
\begin{equation}\label{eq:Poincarecontra4}
\sum_{i,j=1}^n\left(\hat\lambda_i\hat\lambda_j(\hat\lambda_j-\hat\lambda_i)\right)^2=0\,.
\end{equation}
Together, 
\eqref{eq:Poincarecontra2}, \eqref{eq:Poincarecontra3} and \eqref{eq:Poincarecontra4} are in contradiction: \eqref{eq:Poincarecontra4} implies that $\hat\lambda$ has a null component of multiplicity $\ell$ and a non-zero component, $\kappa$ say, of multiplicity $n-\ell$. 
The inequalities \eqref{eq:Poincarecontra2} and \eqref{eq:Poincarecontra3} then yield
\[
\left(n-\ell-\frac{(n-\ell)^2}{n-m+1}\right)\kappa^2\geq \eta(n-\ell)^2\kappa^2>0
\]
and
\[
\left(n-\ell-\frac{(n-\ell)^2}{n-m+\alpha}\right)\kappa^2\leq 0\,.
\]
which together imply that $\ell\in(m-1,m-\alpha]$, which is impossible. This proves \eqref{eq:Poincare1}. 

Using \eqref{eq:Poincare1}, we can estimate
\begin{align*}
\gamma\!\int\! u^2Wd\mu\leq{}& \int  \frac{u^2}{W^{2}}\left(|\mathrm{C}|^2+1\right)d\mu\\
={}& \int \frac{u^2}{W^{2}}\left(\mathrm{C}\ast\nabla^2A+1\right)\,d\mu\\
={}&\int\!\frac{u^2}{W^{2}}\!\left(\frac{\nabla u}{u}\ast\mathrm{C}+\frac{\nabla W}{W}\ast\mathrm{C}+\nabla\mathrm{C}\right)\!\ast\nabla A\,d\mu+\int  \frac{u^2}{W^{2}}\,d\mu\\
\leq{}&C\left[\int\frac{u^2}{W^{2}}\left(W^{\frac{3}{2}}\frac{|\nabla u|}{u}+W^{\frac{1}{2}}|\nabla W|+W|\nabla A|\right)|\nabla A|\,d\mu\right.\\
{}&\qquad\left.+\int u^2\,d\mu\right]\\
\leq{}&C\left[\int u^2\left(\frac{|\nabla u|}{u}+\frac{|\nabla A|}{W^{\frac{1}{2}}}\right)\frac{|\nabla A|}{W^{\frac{1}{2}}}\,d\mu+\int u^2\,d\mu\right],
\end{align*}
where $C$ denotes any constant which depends only on $n$, $\alpha$ and $\eta$. The claim now follows from Young's inequality.
\end{proof}

\begin{lemma}
There exist constants $\ell=\ell(n,\alpha,\eta)<\infty$ and $C=C(n,K,\alpha,V,\Theta,\sigma,p)$ such that
\begin{align}\label{eq:Lpest}
\int f_+^p\,d\mu\leq C\vts\mathrm{e}^{-\delta pKt}
\end{align}
so long as $\sigma\leq \ell p^{-\frac{1}{2}}$ and $p>\ell^{-1}$.
\end{lemma}
\begin{proof}
This follows from Lemma \ref{lem:evol_fsigma} and Proposition \ref{prop:Poincare} exactly as in \cite[Lemma 4.6]{LangfordNguyen}.
\end{proof}

This $L^2$-estimate (for $v:=f_+^{\frac{p}{2}}$) can be bootstrapped to an $L^\infty$-estimate using Stampacchia iteration, exactly as in the proof of \cite[Theorem 4.1]{LangfordNguyen}.

\section{Derivative estimates}

Next, we use the cylindrical estimate to obtain a sharp ``gradient estimate" for the second fundamental form. We need the following universal interior estimates for solutions with initial data in the class $\mathcal{C}^n_m(K,\alpha,V,\Theta)$.

\begin{proposition}\label{prop:class C universal interior estimates}
Let $X:\M^n\times [0,T)\to\mathbb{S}_K^{n+1}$ be a maximal solution to mean curvature flow with initial condition in the class $\mathcal{C}_m^n(K,\alpha,V,\Theta)$. Defining $\Lambda_0$ and $\lambda_0$ by
\begin{equation}\label{eq:lambdas}
\Lambda_0/2:=\frac{1}{n-m+\alpha}\Theta^2+2(m-\alpha)\;\;\text{and}\;\;\mathrm{e}^{2n\lambda_0}:=1+\frac{n}{n+\Lambda_0},
\end{equation}
we have
\begin{equation}\label{eq:class C universal T bound}
\mathrm{e}^{2nKT}\ge 1+\frac{2n}{\Lambda_0}\,,
\end{equation}
and
\begin{equation}\label{eq:class C universal interior estimates}
\max_{\M^n\times\{\lambda_0K^{-1}\}}\vert\cd^k{A}\vert^2\leq \Lambda_kK^{k+1}
\end{equation}
for every $k\in\mathbb{N}$, where $\Lambda_k$ depends only on $n$, $k$ and $\Lambda_0$.
\end{proposition}
\begin{proof}
The proof is similar to \cite[Proposition 4.7]{LangfordNguyen}.
\end{proof}

\begin{theorem}[Gradient estimate (cf. {\cite[Theorem 6.1]{HuSi09}})]\label{thm:gradient estimate}
Let $X:\M^n\times[0,T)\to\mathbb{S}_K^{n+1}$ be a solution to mean curvature flow with initial condition in the class $\mathcal{C}_m^n(K,\alpha,V,\Theta)$. There exist constants $\delta=\delta(n,\alpha)$, $c=c(n,m,\alpha,\Theta)<\infty$, $\eta_0=\eta_0(n,m,\alpha)>0$ and, for every $\eta\in(0,\eta_0)$, $C_\eta=C_\eta(n,\alpha,V,\Theta,\eta)<\infty$ such that
\begin{equation}\label{eq:gradient estimate}
\vert\cd  A\vert^2\leq c\left[(\eta+\tfrac{1}{n-m+1})H^2-\vert{ A}\vert^2\right]W+C_\eta K^2\mathrm{e}^{-2\delta Kt}
\end{equation}
in $\M^n \times [\lambda_0K^{-1},T)$, where $\lambda_0$ is defined by \eqref{eq:lambdas}.
\end{theorem}

Note that the conclusion is not vacuous since, by Proposition \ref{prop:class C universal interior estimates}, the maximal existence time of a solution with initial data in the class $\mathcal{C}_m^n(K,\alpha,V,\Theta)$ is at least $\frac{1}{2nK}\log\left(1+\frac{2n}{\Lambda_0}\right)>\lambda_0 K^{-1}$. 

\begin{proof}[Proof of Theorem \ref{thm:gradient estimate}]
We proceed as in \cite[Theorem 4.8]{LangfordNguyen}, which is inspired by \cite[Theorem 6.1]{HuSi09}. First note that (see, for example, \cite[Inequality (2.13)]{LangfordNguyen})
\bann
(\partial_t- \Delta)|\nabla A|^2 \leq{}& -2 |\nabla^2 A|^2 + c_n\lb |A|^2+nK\rb|\nabla A|^2\,.
\eann
We will control the bad term using the good term in the evolution equation for $\vert A \vert^2$ and the Kato inequality.
 
By the cylindrical estimate, given any $\eta<\eta_0$ we can find a constant $C_\eta=C_\eta(n,\alpha,V,\Theta,\eta)>2$ such that
\[
\vert A\vert^2-\frac{1}{n-m+1}H^2\leq \eta H^2+C_\eta K\mathrm{e}^{-2\delta Kt}\,,
\]
where $\delta=n\eta_0$, and hence
\[
G_\eta:=2C_\eta K\mathrm{e}^{-2\delta Kt}+\left(\eta+\frac{1}{n-m+1}\right)H^2-\vert A\vert^2\geq C_\eta K\mathrm{e}^{-2\delta Kt}>0\,.
\]
Moreover, since $\frac{1}{n-m+\alpha}<\frac{3}{n+2}$,
\[
G_0:=2C_0K+\frac{3}{n+2} H^2-\vert A\vert^2\geq C_0K>0\,,
\]
where $C_0:= 2(m-\alpha)$. By \eqref{eq:evolve_f_eta},
\bann
(\pd_t-\Delta)G_\eta={}&2(\vert A\vert^2+nK)(G_\eta-2C_\eta K\mathrm{e}^{-2\delta Kt})+ 4nK\left(\vert A\vert^2-\tfrac{1}{n}H^2\right)\\
{}&+2\left[\vert\cd  A\vert^2-\lb\eta+\tfrac{1}{n-1}\rb\vert\cd H\vert^2\right]-2\delta C_\eta K^2\mathrm{e}^{-2\delta Kt}\,.
\eann
Since $G_\eta\geq C_\eta K\mathrm{e}^{-2\delta Kt}$, we can estimate $G_\eta-2C_\eta K\mathrm{e}^{-2\delta Kt}\geq -G_\eta$. By the Kato inequality, we can estimate
\bann
\vert\cd  A\vert^2-\lb\tfrac{1}{n-m+1}+\eta\rb\vert\cd H\vert^2\ge{}& \tfrac{n+2}{3}\left[\tfrac{3}{n+2}-\tfrac{1}{n-m+1}-\eta\right]\vert\cd A\vert^2\\
\geq{}& \tfrac{\beta}{2}\vert\cd  A\vert^2\,,
\eann
where
\begin{equation}\label{eq:beta in grad est}
\beta := \frac{1}{2} \left( \frac{ 3}{n+2} - \frac{1}{n-m+1} \right),
\end{equation}
so long as $\eta\leq \left(2-\frac{3}{2(n+2)}\right)\beta$. Estimating, finally, 
\[
2\delta C_\eta K^2\mathrm{e}^{-2\delta Kt}\le 2\delta KG_\eta\,,
\]
we arrive at
\bann
(\pd_t-\Delta)G_\eta\geq{}&-2(\vert A\vert^2+nK)G_\eta+\beta\vert\cd  A\vert^2-2\delta KG_\eta\,.
\eann

Similarly,
\bann
(\pd_t-\Delta)G_0\geq{}&-2(\vert A\vert^2+nK)G_0\,.
\eann

We can now proceed exactly as in \cite[Theorem 4.8]{LangfordNguyen}: We seek a bound for the ratio $\frac{\vert\cd  A\vert^2}{G_\eta G_0}$. Note that, at a local spatial maximum of $\frac{\vert\cd  A\vert^2}{G_\eta G_0}$,
\bann
0=\cd_k\frac{\vert\cd  A\vert^2}{G_\eta G_0}=2\frac{\inner{\cd_k\cd  A}{\cd  A}}{G_\eta G_0}-\frac{\vert\cd  A\vert^2}{G_\eta G_0}\lb\frac{\cd_k G_\eta}{G_\eta}+\frac{\cd_kG_0}{G_0}\rb.
\eann
In particular,
\bann
4\frac{\vert\cd  A\vert^2}{G_\eta G_0}\inner{\frac{\cd G_\eta}{G_\eta}}{\frac{\cd G_0}{G_0}}\leq{}&\frac{\vert\cd  A\vert^2}{G_\eta G_0}\left\vert\frac{\cd G_\eta}{G_\eta}+\frac{\cd G_0}{G_0}\right\vert^2\leq4\frac{\vert\cd^2  A\vert^2}{G_\eta G_0}\,.
\eann
Suppose that $\frac{\vert \cd A\vert^2}{G_\eta G_0}$ attains a parabolic interior local maximum at $(p_0,t_0)$. Then, at $(p_0,t_0)$,
\bann
0\leq{}& (\pd_t-\Delta)\frac{\vert\cd  A\vert^2}{G_\eta G_0}\\
={}&\frac{(\pd_t-\Delta)\vert\cd  A\vert^2}{G_\eta G_0}-\frac{\vert\cd  A\vert^2}{G_\eta G_0}\lb\frac{(\pd_t-\Delta)G_\eta}{G_\eta}+\frac{(\pd_t-\Delta)G_0}{G_0}\rb\\
{}&+\frac{2}{G_\eta G_0}\inner{\cd\frac{\vert\cd  A\vert^2}{G_\eta G_0}}{\cd(G_\eta G_0)}+2\frac{\vert\cd  A\vert^2}{G_\eta G_0}\inner{\frac{\cd G_\eta}{G_\eta}}{\frac{\cd G_0}{G_0}}\\
\leq{}&\frac{\vert\cd  A\vert^2}{G_\eta G_0}\lb (c_n+4)(\vert A\vert^2+nK)+2\delta K-\beta\frac{\vert\cd  A\vert^2}{G_\eta}\rb
\eann
and hence
\bann
\frac{\vert\cd  A\vert^2}{G_\eta G_0}\leq{}& \frac{(c_n+4)(\vert A\vert^2+nK)+2\delta K}{2C_0K+\frac{3}{n+2}H^2-\vert A\vert^2}\,.
\eann
Since
\[
\vert A\vert^2\le \frac{1}{n-m+\alpha}H^2+2(m-\alpha)K\,,
\]
we obtain, at $(p_0,t_0)$,
\bann
\frac{\vert\cd  A\vert^2}{G_\eta G_0}
\le{}&C\,,
\eann
where $C$ depends only on $n$ and $\alpha$.

On the other hand, since $G_0>C_0K$ and $G_\eta>C_\eta K\mathrm{e}^{-2\delta Kt}$, if no interior local parabolic maxima are attained, then, by Proposition \ref{prop:class C universal interior estimates}, we have for any $t\ge \lambda_0K^{-1}$
\bann
\max_{\M^n\times \{t\}}\frac{\vert\cd{A}\vert^2}{G_0G_\eta}\leq{}&\max_{\M^n\times \{\lambda_0K^{-1}\}}\frac{\vert\cd{A}\vert^2}{G_0G_\eta}\\
\leq{}&\max_{\M^n\times \{\lambda_0K^{-1}\}}\frac{\vert\cd{A}\vert^2}{C_0C_\eta K^2\mathrm{e}^{-2\delta \lambda_0}}\\
\leq{}&\frac{\Lambda_1\mathrm{e}^{\lambda_0}}{C_0C_\eta}\\
\leq{}&\Lambda_1\mathrm{e}^{\lambda_0}\,.
\eann
The theorem now follows from Young's inequality.
\end{proof}

\begin{remark}
Fixing $\eta$ in \eqref{eq:gradient estimate} yields the cruder estimate
\begin{equation}\label{eq:gradient estimate eta=1}
\vert\cd  A\vert^2\leq C(H^4+K^2)\,,
\end{equation}
where $C=C(n,\alpha,V,\Theta)$. Indeed, an estimate of this form may actually be obtained without the aid of the cylindrical estimate (simply take $\eta=\frac{1}{n-m+\alpha}-\frac{1}{n-m+1}$, $\delta=0$, and $C_\eta=2(m-\alpha)$ in the proof). A non-quantitative version of the cylindrical estimate without the exponential decay term may then be obtained via a blow-up argument as in  the proof of the convexity estimate (Theorem \ref{thm:convexity estimate}) below (cf. \cite{LyNg}).
\end{remark}

The gradient estimate can be used to bound the first order terms which arise in the evolution equation for $\cd^2{A}$. A straightforward maximum principle argument exploiting this observation yields an analogous estimate for $\cd^2{A}$.

\begin{theorem}[Hessian estimate (cf. \cite{Hu84,HuSi09})]\label{thm:Hessian estimate}
Let $X:\M^n\times[0,T)\to\mathbb{S}_K^{n+1}$ be a solution to mean curvature flow with initial condition in the class $\mathcal{C}_m^n(K,\alpha,V,\Theta)$. There exists $C=C(n,\alpha,V,\Theta)$ such that
\ba\label{eq:scale_invariant_Hessian_estimate}
\vert\cd^2{A}\vert^2\leq C(H^{6}+K^{3})\;\;\text{in}\;\; \M^n \times[\lambda_0K^{-1},T)\,.
\ea
\end{theorem}
\begin{proof}
The proof is the same as that of \cite[Theorem 4.11]{LangfordNguyen}. 
\end{proof}

An inductive argument, exploiting estimates for lower order terms in the evolution equations for higher derivatives of $A$ as in Theorem \ref{thm:Hessian estimate}, can be applied to obtain estimates for spatial derivatives of $A$ to all orders. The evolution equation for $A$ then yields bounds for the mixed space-time derivatives (cf. \cite[Theorem 6.3 and Corollary 6.4]{HuSi09} and \cite[Theorem 4.14]{LangfordNguyen}).

\section{A convexity estimate}

Following White \cite{Wh03}, the gradient estimates allow us to obtain a convexity estimate via a blow-up argument.

\begin{theorem}\label{thm:convexity estimate}
Let $X:\M^n\times[0,T)\to\mathbb{S}_K^{n+1}$ be a solution to mean curvature flow with initial condition in the class $\mathcal{C}_m^n(K,\alpha,V,\Theta)$. For every $\eta>0$ there exists $h_\eta=h_\eta(n,\alpha,V,\Theta,\eta)<\infty$ such that
\begin{align}\label{eq:cylindrical estimate}
\vert H\vert(p,t)\ge h_\eta \sqrt{K}\;\implies\;\lambda_1(p,t)\ge -\eta \vert H\vert(p,t)\,.
\end{align}
\end{theorem}
\begin{proof}
Let $\eta_0$ be the infimum over all $\eta>0$ such that the conclusion holds. By the quadratic pinching hypothesis, $\eta_0<\infty$. Suppose, contrary to the claim, that $\eta_0>0$. Then we can find a sequence of mean curvature flows $X_j:\M_j\times[0,T_j)\to\mathbb S_{K}^{n+1}$ with initial data $X_j(\cdot,0):\M_j\to\mathbb S_{K}^{n+1}$ in the class $\mathcal{C}_m^n(K,\alpha,V,\Theta)$ and a sequence of points $p_j\in \M_j$ and times $t_j\in[0,T_j)$ such that
\[
H_j(p_j,t_j)\to\infty\;\;\text{and}\;\;\frac{\lambda_1^j}{H_j}(p_j,t_j)\to -\eta_0\,.
\]
After translating in time, translating and rotating in $\R^{n+2}$, and parabolically rescaling by $H(p_j,t_j)$, we obtain a sequence of flows $\hat X_j:\M_j\times(-\hat T_j,0]\to (\mathbb S_{K_j}^{n+1}-K_j^{-\frac{1}{2}}e_{n+2})$ such that, at the spacetime origin, $\hat H_j\equiv 1$ and $\hat \lambda_1^j/\hat H_j\to -\eta_0$, where $K_j:= H^{-2}_j(p_j,t_j)K$ and $\hat T_j:=H_j^{2}(p_j,t_j)t_j$. By Proposition \ref{prop:class C universal interior estimates}, $\hat T_j$ is bounded uniformly from below. By the gradient estimates, $\vert \hat A_j\vert$ is bounded uniformly in a backward parabolic cylinder of uniform radius about the spacetime origin. By standard bootstrapping arguments, we obtain a limit flow $X_\infty:\M_\infty\times (-T_\infty,0]\to \R^{n+1}\times \{0\}$ on which $\lambda_1^\infty/H_\infty$ attains a negative local parabolic minimum, $-\eta_0$, at the spacetime origin. But this violates the strong maximum principle for the second fundamental form \cite[Appendix]{Wh03}.
\end{proof}

\section{Singularity models}

To illustrate the utility of our pinching estimates, we will obtain a precise description of singular regions after performing certain types of `blow-ups'. 

Let $X:M\times[0,T)\to \mathbb S_K^{n+1}$ be a maximal solution to mean curvature flow such that $T<\infty$. Standard bootstrapping estimates imply that
\[
\limsup_{t\to T}\max_{\M\times\{t\}}\vert A\vert^2=\infty\,.
\]
Following standard nomenclature, we say that the ``finite time singularity'' is \emph{type-I} if
\[
\limsup_{t\to T}\,(T-t)\max_{\M\times\{t\}}\vert A\vert^2<\infty
\]
and \emph{type-II} otherwise.

Well-known ideas due to Hamilton (see, for example, \cite{HamiltonHarnack,HuSi99a}) yield a classification of certain blow-up sequences. For type-I singularities, we obtain the following.

\begin{theorem}[Finite time type-I singularities]\label{thm:type-I singularities}
Let $X:M\times[0,T)\to \mathbb S_K^{n+1}$ be a solution to mean curvature flow which undergoes a finite time type-I singularity. Suppose that $X(\cdot,0)$ satisfies \eqref{eq:strict quadratic pinching} for some $m$. Given sequences of times $t_j\in [0,T)$ and points $p_j\in \M$ such that $\limsup_{j\to\infty}\vert A(p_j,t_j)\vert\to\infty$, the sequence of solutions $X_j:\M\times [-r_j^{-2}t_j,1)\to r_j^{-1}(\mathbb S_{K_j}-e_{n+2})\subset \R^{n+2}$ obtained by translating the point $(X(p_j,t_j),t_j)$ to the spacetime origin in $\R^{n+2}\times\R$, rotating so that $T_{0}(\mathbb S_{K}^{n+1}-X(p_j,t_j))$ becomes $\R^{n+1}\times\{0\}$ and parabolically rescaling by $r_j^{-1}:=(T-t_j)^{-\frac{1}{2}}$, converge locally uniformly in the smooth topology, after passing to a subsequence and performing a fixed rotation, to a maximal, locally convex, type-I ancient solution $X:(\R^k\times \M^{n-k}_\infty)\times(-\infty,1)\to\R^{n+1}$ to mean curvature flow in $\R^{n+1}$, $k<m$, which satisfies
\[
\vert\cd  A\vert^2\leq c\left(\tfrac{1}{n-m+1}H^2-\vert{ A}\vert^2\right)H^2\,,
\]
where $c$ depends only on $X(\cdot,0)$ and the restriction $X|_{\{0\}\times\M^{n-k}_\infty}$ is locally uniformly convex and solves mean curvature flow in $\{0\}\times \R^{n-k+1}$. If $k=m-1$, then $X|_{\{0\}\times\M^{n-k}_\infty}$ is totally umbilic, else
\[
\vert A\vert^2<\tfrac{1}{n-m+1}H^2\,.
\]
\end{theorem}
\begin{proof}
Using Theorems \ref{thm:cylindrical estimate}, \ref{thm:gradient estimate} and \ref{thm:convexity estimate}, the conclusion is obtained as in \cite[Section 4]{HuSi99a} (strict inequality in the quadratic pinching when $k<m-1$ follows from the strong maximum principle).
\end{proof}

\begin{remark}
In fact, in Theorem \ref{thm:type-I singularities}, the blow-up is a shrinking round orthogonal cylinder for \emph{every} $k$. Indeed, using Hamilton's extension of Huisken's monotonicity formula to general ambient spaces \cite{HamMono} and his matrix Harnack estimate for the heat equation \cite{HamiltonMatrixHarnack}, we may proceed as in \cite{Hu90} to show that the the blow-up is a Euclidean self-shrinking solution. By Theorem \ref{thm:convexity estimate}, the limit is weakly convex. The conclusion then follows from Huisken's classification result \cite{Hu93}. We believe that it should be possible to reach this conclusion without using the monotonicity formula, however.
\end{remark}

For type-II singularities, we obtain the following.

\begin{theorem}[Finite time type-II singularities]
Let $X:M\times[0,T)\to \mathbb S_K^{n+1}$ be a solution to mean curvature flow which undergoes a finite time type-II singularity. Suppose that $X(\cdot,0)$ satisfies \eqref{eq:strict quadratic pinching} for some $m$. There are sequences of times $t_j\in [0,T)$, points $p_j\in \M$ and scales $r_j\to 0$ such that the sequence of solutions $X_j:\M\times [-r_j^{-2}t_j,r_j^{-2}(T-t_j-j^{-1}))\to r_j^{-1}(\mathbb S_{K_j}-e_{n+2})\subset \R^{n+2}$ obtained by translating $(X(p_j,t_j),t_j)$ to the spacetime origin in $\R^{n+2}\times\R$, rotating so that $T_{0}(\mathbb S_{K}^{n+1}-X(p_j,t_j))$ becomes $\R^{n+1}\times\{0\}$ and parabolically rescaling by $r_j^{-1}$, converge locally uniformly in the smooth topology, after passing to a subsequence and performing a fixed rotation, to a locally convex eternal solution $X:(\R^k\times \M^{n-k}_\infty)\times(-\infty,0)\to\R^{n+1}$ to mean curvature flow in $\R^{n+1}$, $k<m-1$, which satisfies
\[
\vert A\vert^2<\tfrac{1}{n-m+1}H^2\;\;\text{and}\;\; \vert\cd  A\vert^2\leq c\left(\tfrac{1}{n-m+1}H^2-\vert{ A}\vert^2\right)H^2\,,
\]
where $c$ depends only on $X(\cdot,0)$ and the restriction $X|_{\{0\}\times\M^{n-k}_\infty}$ is a locally uniformly convex translating solution to mean curvature flow in $\{0\}\times \R^{n-k+1}$. If $k=m-2$, then $X|_{\{0\}\times\M^{n-k}_\infty}$ is the bowl soliton.
\end{theorem}
\begin{proof}
Using Theorems \ref{thm:cylindrical estimate}, \ref{thm:gradient estimate} and \ref{thm:convexity estimate}, the conclusion is obtained by proceeding as in \cite[Section 4]{HuSi99a} and applying the classification result from \cite{BL17} (strict inequality in the quadratic pinching on the limit follows from the strong maximum principle).
\end{proof}

\bibliographystyle{plain}
\bibliography{../bibliography}

\end{document}